\documentclass[12pt]{amsart}

\usepackage[utf8]{inputenc}

\usepackage{amsmath}
\usepackage{amsfonts,amssymb}
\usepackage{xypic}

\theoremstyle{plain}
\newtheorem*{theorem*}{Theorem}
\newtheorem*{lemma*} {Lemma}
\newtheorem*{corollary*} {Corollary}
\newtheorem*{proposition*} {Proposition}
\newtheorem{theorem}{Theorem}[section]
\newtheorem{lemma}[theorem]{Lemma}

\newtheorem{proposition}[theorem]{Proposition}

\theoremstyle{remark}

\newtheorem*{claim}{Claim}

\theoremstyle{definition}

\def\sl{\op{SL}}
\def\op{\operatorname}

\def\co{\colon}

\def\s{\sigma}
\def\gl{\op{GL}}

\newcommand{\R}{\mathbb{R}}

\newcommand{\Z}{\mathbb{Z}}
\newcommand{\C}{\mathbb{C}}

\def\id{\mbox{id}}

\def\wti{\widetilde}

\def\part{\partial}
\def\ll{\langle}
\def\rr{\rangle}

\def\a{\alpha}

\def\tor{\mbox{Tor}}
\def\bp{\begin{pmatrix}}

\def\sm{\setminus}
\def\ep{\end{pmatrix}}
\def\bn{\begin{enumerate}}

\def\en{\end{enumerate}}
\def\ba{\begin{array}}
\def\ea{\end{array}}

\def\s{\sigma}

\def\ti{\tilde}

\def\fr12{\frac{1}{2}}

\def\ker{\mbox{Ker}}

\def\hom{\mbox{Hom}}

\def\aut{\mbox{Aut}}

\def\tpm{[t^{\pm 1}]}

\def\be{\begin{equation}}
\def\ee{\end{equation}}

\def\ct{\C\tpm}

\def\ol{\overline}    

\begin{document}

\title{$3$-manifolds that can be made acyclic}

\author{Stefan Friedl}
\address{Fakult\"at f\"ur Mathematik\\ Universit\"at Regensburg\\93040 Regensburg\\   Germany}
\email{sfriedl@gmail.com}

\author{Matthias Nagel}
\address{Fakult\"at f\"ur Mathematik\\ Universit\"at Regensburg\\93040 Regensburg\\   Germany}
\email{matthias.nagel@mathematik.uni-regensburg.de}
\date{\today}

\begin{abstract}
We determine which  $3$-manifolds admit a unitary representation such that the corresponding twisted chain complex is acyclic.
\end{abstract}

\maketitle

\section{Introduction}
Let $N$ be a $3$-manifold. Here and throughout the paper we assume that all 3-manifolds are compact, connected and orientable.
We say that a representation $\a\co \pi_1(N)\to \gl(k,\C)$ \emph{makes $N$ acyclic} if the twisted homology groups
$H_k(N;\C^k_\a)$ all vanish.

It is clear that $S^3$ does not admit a representation that makes it acyclic since the 0-th homology will always be non-trivial
On the other hand, any lens space $L(p,q)$ can be made acyclic using any non-trivial one-dimensional representation.
This fact played a major r\^ole in Reidemeister's classification of lens spaces \cite{Re35} using torsion invariants.
If $N$ is a finite volume hyperbolic $3$-manifold, then it follows from Raghunathan's vanishing theorem \cite{Ra65}
and from work of Menal-Ferrer  and Porti \cite[Section~2]{MFP14} that there exists a lift $\pi_1(N)\to \sl(2,\C)$ of the holonomy representation that makes $N$ acyclic.
We refer to Section \ref{section:acyclicsl2c} for more details. 

In this paper we  give a complete answer to the question:  which  $3$-manifolds can be made acyclic using unitary representations? Before we state our main result we recall that the Prime Decomposition Theorem of Kneser--Milnor \cite{Kn29,Mi62} says that any $3$-manifold admits a unique decomposition $N\cong N_1\#\dots \#N_k$ as a connected sum of prime $3$-manifolds $N_1,\dots,N_k$. 

Using recent results of Agol \cite{Ag08,Ag13}, Liu \cite{Liu13}, Przytycki--Wise \cite{PW14,PW12}, Wise \cite{Wi09,Wi12a,Wi12b} and also of the second author \cite{Na14} we will prove the following theorem.

\begin{theorem}\label{mainthm}
Given a  $3$-manifold $N\ne S^3$ the following statements are equivalent:
\bn
\item $N$ can be made acyclic using a unitary representation,
\item $N$ is closed or has toroidal boundary
and $N$ has at most one prime summand that is not a rational homology sphere.
\en
\end{theorem}
 
The theorem says in particular that a finite volume hyperbolic 3-manifold admits a unitary representation that makes it acyclic. In fact, we will see in the proof of Theorem \ref{mainthm} that the unitary representation can be chosen such that it  factors through a finite group. This representation
is thus quite different from a lift of the holonomy representation.

The paper is organized as follows. In Section \ref{section:twihom} we recall the definition of twisted homology and cohomology and we recall several of the key properties of twisted invariants.
In Section \ref{section:acyclicsl2c} we show that the results of  Raghunatha  \cite{Ra65} and Menal-Ferrer  and Porti \cite[Section~2]{MFP14} imply that hyperbolic 3-manifolds can be made acyclic using a lift of the holonomy representation. In Section \ref{section:12} we show the (1) $\Rightarrow$ (2) implication of the main theorem, and we show the reverse implication in Section \ref{section:21}.

\subsection*{Conventions.}
Throughout this paper all 3--manifolds are understood to be  compact,
connected and orientable. Furthermore all groups are understood to be finitely generated.
Also, all topological spaces are assumed to be path connected and locally path connected. In particular, this ensures that the universal cover is always defined.

\subsection*{Acknowledgement.}
The first author wishes to thank the roads of Newfoundland for giving him ample time to think about mathematics.
Both authors gratefully acknowledge the support provided by the SFB 1085 `Higher Invariants' at the University of Regensburg, funded by the Deutsche Forschungsgemeinschaft DFG. 

\section{Twisted homology and cohomology groups}\label{section:twihom}

\subsection{Unitary representations}
Let $R$ be a commutative ring.
An \emph{$R$-representation} (or short \emph{representation}) for a group $\pi$ is a finite-dimensional  $R$-module $V$ together with a homomorphism $\a\colon \pi\to \aut_R(V)$.  Equivalently, an \emph{$R$-representation} is a $(\Z[\pi],R)$-bimodule.
Sometimes, when we want to specify the action of $\pi$ on $V$ we denote the  $(\Z[\pi],R)$-bimodule by $V_\a$.

We  say that a complex representation $\a\colon\pi\to \aut_\C(V)$ is \emph{unitary} if $V$ admits a positive-definite hermitian inner product with respect to which $\pi$ acts via isometries.  
By Sylvester's theorem all positive-definite hermitian inner products  on $\C^n$ are isometric. It thus follows that given any $n$-dimensional unitary representation $\a\co \pi\to \aut_\C(V)$ there exists an isomorphism $f\co V\to \C^n$ such that $f\circ \a(g)\circ f^{-1}\in U(n)$ for all $g\in \pi$. Conversely, a homomorphism $\a\co \pi\to U(n)$ evidently gives rise to a unitary  representation in the above sense.

Let $\pi$ be a group and let $\pi'\subset \pi$ be a  subgroup. If $V$ is a $(\Z[\pi'],\C)$-bimodule,
then we endow $\Z[\pi]\otimes_{\Z[\pi']}V$ with the $\Z[\pi]$-left module structure given by left multiplication.
We can thus view $\Z[\pi]\otimes_{\Z[\pi']}V$ as a $(\Z[\pi],\C)$-bimodule and we refer to it as the \emph{induced  $(\Z[\pi],\C)$-bimodule}. 

We have the following elementary lemma:

\begin{lemma}\label{lem:induced}
Let $\pi$ be a group, let $\pi'\subset \pi$ be a finite-index  subgroup and let $V'$ be a $(\Z[\pi'],\C)$-bimodule.
We denote by $V=\Z[\pi]\otimes_{\Z[\pi']}V'$ the induced $(\Z[\pi],\C)$-bimodule.
Then the following statements hold:
\bn
\item We have 
\[ \dim_\C(V)=[\pi:\pi']\cdot \dim_\C(V').\]
\item If $\pi'\to \aut_\C(V')$ factors through a finite group, then so does $\pi\to \aut_\C(V)$.
\item If  $\pi'\to \aut_\C(V')$ is a unitary representation, then so is  $\pi\to \aut_\C(V)$.
\en
\end{lemma}

The lemma is well-known. Therefore we just give a quick sketch of the argument.

\begin{proof}
We write $d=[\pi:\pi']$ and we pick coset representatives $g_1,\dots,g_d$ of $\pi/\pi'$. Furthermore we pick a basis $v_1,\dots,v_n$ for the complex vector space $V'$. 
\bn
\item It is straightforward to see that $g_i\otimes v_j$, $i\in \{1,\dots,d\}$ and $j\in \{1,\dots,n\}$ is a basis for the complex vector space $V:=\Z[\pi]\otimes_{\Z[\pi']}V'$.
\item We denote by $\pi'':=\cap_{g\in \pi}g\pi'g^{-1}$ the core of $\pi'$. Since $\pi'$ is finite-index in $\pi$ it follows from standard arguments that $\pi''$ is finite-index in $\pi$. If $\Gamma':=\ker(\pi'\to \aut_\C(V'))$ has finite-index in $\pi'$, then it is also straightforward to see that  $\Gamma'\cap \pi''$ has finite-index in $\pi$ and that $\Gamma'\cap \pi''$ is contained in $\ker(\pi\to \aut_\C(V'))$.
\item If $\ll\,,\,\rr$ denotes a positive-definite hermitian inner product on $V'$ that is preserved by $\pi'$, then it is straightforward to verify that 
\[  \ba{ccl} \Z[\pi]\otimes_{\Z[\pi']}V'\times \Z[\pi]\otimes_{\Z[\pi']}V'&\to & \C\\
\left(\sum_{i=1}^d g_i\otimes s_i,\sum_{i=1}^d g_i\otimes t_i \right)&\mapsto & \sum_{i=1}^d\ll s_i,t_i\rr\ea \]
is a positive-definite hermitian inner product on $V$ that is preserved by the $\pi$-action.
\en
\end{proof}

\subsection{The definition of twisted homology and cohomology groups}

Let $X$ be a  topological space and let $Y\subset X$ be a subset.  We write $\pi=\pi_1(X)$. 
Let $R$ be a commutative ring and let   $\a\colon \pi\to \aut_R(V)$ be an $R$-representation.
We denote the universal covering of $X$ by $p\co \ti{X}\to X$ and we write $\ti{Y}:=p^{-1}(Y)$.
There exists a canonical left $\pi$--action on the universal cover $\ti{X}$ given by deck transformations
and we can therefore consider
the  complex
\[ C^*(X,Y;V):=\hom_{\Z[\pi]}(C_*(\tilde{X},\ti{Y};\Z),V)\]
of $R$-modules
and the \emph{twisted cohomology modules}
\[
H^i(X,Y;V) := H_i(\hom_{\Z[\pi]}(C_*(\tilde{X},\ti{Y};\Z),V)).\]

We
can also view the cellular chain complex $C_*(\ti{X},\ti{Y};\Z)$  as a right $\Z[\pi]$-module by defining $\s \cdot
g:=g^{-1}\s$ for a chain $\s$ and an element $g\in \pi$. We then consider the chain complex
\[ C_*(X,Y;V):=C_*(\tilde{X},\ti{Y};\Z)\otimes_{\Z[\pi]}V\]
of $R$-modules 
and the \emph{twisted homology modules}
\[
H_i(X,Y;V) := H_i(C_*(\tilde{X},\ti{Y};\Z)\otimes_{\Z[\pi]}V).\]
As usual we will drop $Y$ from the notation when $Y=\emptyset$. 

\subsection{Basic results about twisted homology and cohomology groups}
Throughout this section $K$ denotes a field. Given a finite CW-complex $X$ and a representation $\a\colon \pi_1(X)\to \aut_K(V)$ we write
\[ \chi(X,\a):=\sum_{i} (-1)^{i} \dim H_i(X;V).\]
We start out with the following well-known lemma.

\begin{lemma}\label{lem:eulermultiplicative}
Let $X$ be  a finite CW-complex $X$ and let  $\a\colon \pi_1(X)\to \aut_K(V)$ be a representation. Then
\[ \chi(X,\a)=\dim V\cdot \chi(X).\]
\end{lemma}

\begin{proof}
We have 
\[ \ba{rcl} \chi(X,\a)&=&
\sum_{i} (-1)^{i} \dim H_i(X;V)\\[2mm]
&=&\sum_{i} (-1)^{i} \dim C_i(X;V)\\[2mm]
&=&\dim V\cdot \sum_{i} (-1)^{i} \dim C_i(X;\C)=\dim V\cdot \chi(X).\ea\]
\end{proof}

The following lemma follows easily from the definitions. 

\begin{lemma}\label{lem:trivialrep}
Let $X$ be a  topological space and let  $\a\colon \pi_1(X)\to \aut_{\C}(V)$ be a trivial complex representation. Then for any $i$ we have an isomorphism
\[ H_i(X;V)\cong H_i(X;\Z)\otimes_{\Z}V.\]
\end{lemma}

The next lemma says in particular that the zeroth and first twisted homology groups only depend on the fundamental group and the representation. The proof of the lemma is also standard, and left to the reader.

\begin{lemma}\label{lem:dependsonlyonpi1}
Let $X$ be a finite CW-complex. There exists a $2$-connected map 
$f \colon X \rightarrow K(\pi_1(X),1)$
to the Eilenberg-Maclane space which induces the identity on fundamental groups.
Given a representation $\a\colon \pi_1(X)\to \aut_K(V)$ 
such a map $f$ will induce isomorphisms
\[ H_i(X; V_{\a}) \cong H_i( \pi_1(X); V_{\a}) \]
with $i = 0,1$. The right-hand side depends only on the fundamental group of $X$
and the representation $\a$.
\end{lemma}

Given a topological space $X$, a representation $\a\colon \pi_1(X)\to \aut_K(V)$ and a connected subspace $Y$ of $X$  a choice of a path connecting the base points induces a representation $\pi_1(Y)\to \pi_1(X)\to \aut_K(V)$ that we denote again by $\a$. Furthermore the inclusion map $Y\to X$ induces for every $i$ a map 
\[ H_i(Y;V_\a)\to H_i(X;V_\a).\]
We refer to \cite[Section~2.1]{FKm06} for details. This map  depends on the choice of a path connecting the base points. This indeterminacy can make a big difference in some circumstances, see \cite[Section~3]{FKm06}, but it will not affect us. We will therefore suppress this issue from our discussions and from our notation.

We continue with the following well-known lemma, see \cite[Section~VI.3]{HS97} for details.

\begin{lemma}\label{lem:h0}
Let $X$ be a finite CW-complex and let  $\a\colon \pi_1(X)\to \aut_K(V)$ be a $K$-representation. Then 
\[ H_0(X;V_\a)\cong  V/\{\sum_i \a(g_i)(v_i)-v_i\,|\, g_i\in \pi_1(X), v_i\in V\}. \]
Furthermore, if $Y$ is a connected subcomplex of $X$, then the inclusion  maps induce a commutative diagram
\[  \xymatrix@C2cm{H_0(Y;V_\a)\ar[d]\ar[r]^-\cong &  V/\{\sum_i \a(g_i)(v_i)-v_i\,|\, g_i\in \pi_1(Y), v_i\in V\}\ar[d]\\
 H_0(X;V_\a)\ar[r]^-\cong & V/\{\sum_i \a(g_i)(v_i)-v_i\,|\, g_i\in \pi_1(X), v_i\in V\}.} \]
\end{lemma}

We will also need the following lemma.

\begin{lemma} \label{lem:les}
Let  $X$ be a   topological space and let
\[ 0\to W\to V\to V/W\to 0\]
be a short exact sequence of $(\Z[\pi_1(X)],K)$-modules. Then there exists a long exact sequence 
of $K$-modules
\[ \dots \to H_i(X;W)\to H_i(X;V)\to H_i(X;V/W)\to H_{i-1}(X;W) \to \dots \]
\end{lemma}

\begin{proof}
The long exact sequence of the lemma is  the long exact sequence in homology corresponding to the short exact sequence 
\[ 0\to C_*(X;W)\to C_*(X;V)\to C_*(X;V/W)\to 0\]
of chain complexes.
\end{proof}

The following lemma is sometimes referred to as the Eckmann-Shapiro lemma. 

\begin{lemma}\label{lem:eckmann-shapiro}
Let $p\colon X'\to X$ be a connected finite covering of a  topological space. Let $\a\co \pi_1(X')\to \aut_\C(V')$ be a representation.
Then for any $i$ there is an isomorphism
\[  H_i(X';V)\to H_i(X;\Z[\pi_1(X)]\otimes_{\Z[\pi_1(X')]}V').\]
\end{lemma}

\begin{proof}
We denote by $\wti{X}$ the universal cover of $X$ which is of course also the universal cover of $X'$. 
The canonical isomorphism
\[  C_*(\wti{X})\otimes_{\Z[\pi_1(X')]} V'\to   C_*(\wti{X})\otimes_{\Z[\pi_1(X)]}\Z[\pi_1(X)]\otimes_{\Z[\pi_1(X')]}V'\]
then descends to the desired isomorphism of twisted homology groups.
\end{proof}

%

\section{$\sl(2,\C)$-Representations which make a manifold acyclic}\label{section:acyclicsl2c}
In this section we will see, as claimed in the introduction, that it follows from the results of Raghunathan \cite{Ra65} and  Menal-Ferrer  and Porti \cite[Section~2]{MFP14} that any finite-volume hyperbolic 3-manifold admits an $SL(2,\C)$-representation that makes it acyclic. This result is not needed for the proof of Theorem \ref{mainthm}. We only provide the argument for completeness' sake.

First we recall the following theorem of  Raghunathan \cite{Ra65} and Menal-Ferrer  and Porti \cite{MFP12,MFP14}.

\begin{theorem}\label{thm:ramfp}
Let $N$ be a finite-volume hyperbolic 3-manifold.
\bn
\item If $N$ is closed and if  $\a\co\pi_1(N)\to \sl(2,\C)$ is any lift of the holonomy representation, then $H^i(N;\C^2_\a)=0$ for all $i$.
\item If $N$ has non-empty boundary then there exists a lift   $\a\co\pi_1(N)\to \sl(2,\C)$ of the holonomy representation such that $H^i(N;\C^2_\a)=0$
and $H^i(\partial N;\C^2_\a)=0$ for all $i$.
\en
\end{theorem}

\begin{proof}
The first statement is an immediate consequence of Raghunathan's vanishing Theorem \cite{Ra65}, see also \cite[Corollary~2.2]{MFP12} for more details.

Now let $N$ be a finite-volume hyperbolic 3-manifold with non-empty boundary.
Menal-Ferrer  and Porti \cite[Section~3]{MFP14} show that $N$ admits an `acyclic' spin-structure. This means that the lift $\a$ of the holonomy representation corresponding to such a spin-structure satisfies
 $H^i(\partial N;\C^2_\a)=0$ for all $i$. By  
    \cite[Theorem~0.1]{MFP12} we then also have $H^i(N;\C^2_\a)=0$.
\end{proof}

Note that this theorem is a vanishing result for twisted cohomology. The following lemma turns this into  a vanishing result for twisted homology groups.

\begin{lemma}
Let $N$ be a 3-manifold and let  $\a\co \pi_1(N)\to \sl(2,\C)$ be a representation such that all the twisted cohomology groups $H^i(N;\C_\a^2)$ and $H^i(\partial N;\C_\a^2)$ vanish. Then all the twisted homology groups $H_i(N;\C_\a^2)$ also vanish.
\end{lemma}

\begin{proof}
We first note that any matrix in $\sl(2,\C)$ preserves the non-singular bilinear form on $\C^2$ given by the matrix 
\[ \bp 0&1\\ -1&0\ep.\]
It then follows from Poincar\'e duality for twisted homology and cohomology (see e.g.\ \cite{FKm06} and \cite[Section~3]{HSW10}) that $H_i(N;\C^2_{\a})\cong H^{3-i}(N,\partial N;\C^2_\a)$ for all $i$.
The lemma now follows from the long exact sequence in cohomology which implies that if all the twisted cohomology groups $H^i(N;\C_\a^2)$ and $H^i(\partial N;\C_\a^2)$ vanish, then the cohomology groups $H^i(N,\partial N;\C_\a^2)$ also vanish.
\end{proof}

\section{Proof of the main theorem}

\subsection{Proof of the main theorem: $(1) \Rightarrow (2)$}\label{section:12}

The goal of this section is to prove the implication $(1) \Rightarrow (2)$
in Theorem \ref{mainthm}. Put differently, the goal is to prove the following two propositions.

\begin{proposition}\label{mainprop23a}
If a $3$-manifold  has a  boundary component that is not a torus,  then it can not be made acyclic using a unitary representation.
\end{proposition}

\begin{proposition}\label{mainprop23b}
If  the connected sum of two $3$-manifolds $N\# N'$ can be made acyclic using a unitary representation, then one of the two summands is a rational homology sphere.
\end{proposition}

Before we give a proof of Proposition \ref{mainprop23a}, we recall the following well-known lemma.

\begin{lemma}\label{lem:euler}
Let $M$ be a $3$-manifold, then $\chi(M)=\frac{1}{2}\chi(\partial M)$.
\end{lemma}

\begin{proof}
It is a consequence of Poincar\'e duality and the universal coefficient theorem that 
$\chi(M,\partial M)=-\chi(M)$. It  follows from the definitions that $\chi(M)=\chi(M,\partial M)+\chi(\partial M)$. Combining these two equalities gives the desired statement.
\end{proof}

Now we turn to the proof of Proposition \ref{mainprop23a}.

\begin{proof}[Proof of Proposition \ref{mainprop23a}]
Let  $N$ be a $3$-manifold that has a  boundary component that is not a torus. Let $\a\co \pi_1(N)\to \aut_\C(V)$ be a representation.
We first consider the case that $N$ has no boundary component of negative Euler characteristic. This implies that all boundary components of $N$ are tori and spheres, and by assumption $N$ has at least one boundary component that is a sphere. Thus it follows that $\chi(\partial N)>0$, which by Lemma \ref{lem:euler} implies that $\chi(N)>0$. 
This in turn implies by Lemma \ref{lem:eulermultiplicative} that 
$\chi(\widehat{N},\a)>0$. In particular we see that $\a$ does not make $N$ acyclic.

We now turn to the case that $N$ has at least one boundary component of negative Euler characteristic.
We denote by $\widehat{N}$ the $3$-manifold that is given by capping off all spherical boundary components of $N$ with a $3$-ball.
The inclusion $N\to \widehat{N}$ induces an isomorphism of fundamental groups and we equip $\widehat{N}$ with the representation $\pi_1(\widehat{N})\xleftarrow{\cong} \pi_1(N)\to \gl(k,\C)$. By Lemma \ref{lem:dependsonlyonpi1} we have $H_1({N};V)\cong H_1(\widehat{N};V)$.

By assumption the manifold $N$, and thus also $\widehat{N}$, has a boundary component of negative Euler characteristic. Since $\widehat{N}$ has no boundary component of positive Euler characteristic it follows that $\chi(\partial \widehat{N})<0$. As above this implies that 
$\chi(\widehat{N},\a)<0$. Since $\widehat{N}$ has non-trivial boundary, it is homotopy equivalent to a 2-complex. This implies that $H_i(\widehat{N};V)=0$ for $i\geq 3$. It thus follows from 
$\chi(\widehat{N},\a)<0$ that $\dim H_1(\widehat{N};V)>0$.
By the above this implies that also $\dim H_1({N};V)>0$. 
\end{proof}

\begin{proof}[Proof of Proposition \ref{mainprop23b}]
Let $N\# N'$ be a connected sum of two $3$-manifolds $N$ and $N'$  that can be made acyclic using a unitary representation $\a\co \pi_1(N\# N')\to \aut_\C(V)$. In the following we denote by $D\subset N$ and $D'\subset N'$  two open 3-balls and we identify $N\# N'$ with $(N\sm D)\cup_{S^2} (N'\sm D')$. 
By Proposition \ref{mainprop23a} we already know that the boundary of $N\# N'$ is either closed or toroidal.
Therefore the same holds for $N$ and $N'$.

Since $H_*(N\# N';V)=0$ it follows from the Mayer-Vietoris sequence for homology with twisted coefficients corresponding to the decomposition $(N\sm D)\cup_{S^2} (N'\sm D')$ that for any $i$ we have isomorphisms
\be \label{equ:iso1} H_i(S^2;V) \xrightarrow{\cong} H_i(N\sm D;V) \oplus H_i(N'\sm D';V).\ee
Furthermore it follows from Lemma \ref{lem:dependsonlyonpi1} that for  $i=0,1$  we have isomorphisms
\be \label{equ:iso2} H_i(N\sm D;V)\cong H_i(N;V)\mbox{ and } H_i(N'\sm D';V)\cong H_i(N';V).\ee
By Lemma \ref{lem:trivialrep} we have  $ H_1(S^2;V)=0$ and $ H_0(S^2;V)\cong V$.
It follows from (\ref{equ:iso1}) and (\ref{equ:iso2}) that $H_1(N;V)=H_1(N';V)=0$ and  that at least one of 
$H_0(N;V)$ and $H_0(N';V)$ has dimension at least $\frac{1}{2}\dim V$. 
Without loss of generality we can assume that $\dim H_0(N;V)\geq \frac{1}{2}\dim V$. 
We will show that $N$ is in fact a rational homology sphere.

We write 
\[ W= \left\{ \sum_{j=1}^n \a(g_j)v_j-v_j\,|\,g_1,\dots,g_n\in \pi_1(N)\mbox{ and }v_1,\dots,v_n\in V\right\}.\]
 We continue with the following claim:

\begin{claim}\mbox{}
\bn
\item The subspace  $W$ is  a $(\pi_1(N),\C)$-submodule of $V$.
\item The subspace $W^\perp\subset V$ is a  $(\pi_1(N),\C)$-submodule of $V$ such that $\pi_1(N)$ acts trivially on $W^\perp$.
\item The projection map $V\to V/W$ descends to an isomorphism $W^\perp\to V/W$.
\en
\end{claim}

We first note that for any $h,g\in \pi_1(N)$ and $v\in V$ we have 
\[ h\cdot (\a(g)v-v)=\a(hg)v-\a(h)v=(\a(hg)v-v)+(\a(h)(-v)+(-v)).\]
This implies that  $W$ is indeed a $(\pi_1(N),\C)$-submodule of $V$.
Since $\pi_1(N)$ acts isometrically on $V$ it is also straightforward to see  that  $W^\perp$ is  a $(\pi_1(N),\C)$-submodule of $V$. It remains to show that  the $\pi_1(N)$-action on $W^\perp$ is trivial.
To show this, let  $w\in W^\perp$ and $g\in \pi_1(N)$. We then have 
\[ \ll \a(g)w-w,\a(g)w-w\rr=\ll \a(g)w,\a(g)w-w\rr-\ll w,\a(g)w-w\rr=0.\]
Since $\ll \,,\,\rr$ is positive-definite it follows that $\a(g)(w)-w=0$, i.e.\ $\a(g)(w)=w$. 
Finally the third statement of the claim is obvious.
This concludes the proof of the claim.

It follows from Lemma \ref{lem:les}, from $H_1(N;V)=0$  and from $V/W\cong W^\perp$ that there exists an exact sequence
\be \label{equ:les4} 0 \to  H_1(N;W^\perp)\to  H_0(N;W)\to H_0(N;V)\to H_0(N;W^\perp)\to 0.\ee
Since $\pi_1(N)$ acts trivially on $W^\perp$ it follows from  Lemma \ref{lem:trivialrep} that  $ H_1(N;W^\perp)\cong (W^\perp)^{b_1(N)}$ and $ H_0(N;W^\perp)\cong W^\perp$.
By Lemma \ref{lem:h0} we have $H_0(N;V)\cong V/W\cong W^\perp$, in particular the map on the right in the  exact sequence (\ref{equ:les4}) is an isomorphism. Hence the map on the left of (\ref{equ:les4}) also needs to be an isomorphism. We conclude that 
\[ b_1(N)\cdot \dim W^\perp =\dim H_1(N;W^\perp)=\dim H_0(N;W)\leq \dim W.\]
On the other hand we also have the following inequality
\[\frac{1}{2}\dim V\leq  \dim H_0(N;V)=\dim (V/W) =\dim W^\perp.\]
Together with $\dim W+\dim W^\perp=\dim V$ we also get the inequality
\[ \dim W\leq  \frac{1}{2}\dim V.\]
Putting these three inequalities together we see that 
\[ b_1(N)\cdot \dim W^\perp\leq \dim H_0(N;W)\leq \dim W\leq \dim W^\perp.\]
This implies that $b_1(N)=0$ or $\dim H_0(N;W)=\dim W$.

We first consider the case that $b_1(N)=0$. As we mentioned above, we know that $N$ is either closed or has toroidal boundary. An elementary Poincar\'e duality and Euler characteristic argument shows that $b_1(N)=0$ implies that $b_2(N)=0$ and that $N$ is closed. Put differently, $N$ is indeed a rational homology sphere.

Now we consider the case that $\dim H_0(N;W)=\dim W$. In that case it follows from Lemma \ref{lem:h0} that $\pi_1(N)$ acts trivially on $W$. But we already showed that $\pi_1(N)$ also acts trivially on $W^\perp$. Since $\pi_1(N)\to \aut(W\oplus W^\perp)=\aut(V)$ is unitary it follows that the representation $\pi_1(N)\to \aut(V)$ is trivial. But it then follows from (\ref{equ:iso1}) that $b_1(N)=0$. As we have seen above, this implies that $N$ is a rational homology sphere.
\end{proof}

\subsection{Proof of the main theorem: $(2) \Rightarrow (1)$}\label{section:21}
In this section we will prove the implication   $(2) \Rightarrow (1)$ of Theorem \ref{mainthm}. In fact we will prove the following slightly stronger statement:

\begin{theorem}\label{mainthm21}
Let $N\ne S^3$ be a $3$-manifold that is closed or has toroidal boundary such that  $N$ has at most one prime summand that is not a rational homology sphere. Then $N$ can be made acyclic using a unitary representation that factors through a finite group.
\end{theorem}

The proof of Theorem \ref{mainthm21} will require the remainder of this section and it will be broken up into several special cases. In most of these special cases we will make use of the following lemma.

\begin{lemma}\label{lem:gotofinitecover}
Let $N$ be a $3$-manifold which admits a finite cover  that can be made acyclic using a unitary representation that factors through a finite group. Then $N$ can also be made acyclic using a unitary representation that factors through a finite group.
\end{lemma}

\begin{proof}
Let  $p\colon N'\to N$ be a finite cover and let  $\a'\colon \pi_1(N')\to \aut_\C(V')$ be  a unitary representation that factors through a finite group that makes $N'$ acyclic. We denote by $\a\co \pi_1(N)\to \aut_\C(\Z[\pi_1(N)]\otimes_{\Z[\pi_1(N')]}V')$ the induced representation. It follows from Lemmas \ref{lem:induced} and \ref{lem:eckmann-shapiro} that $\a$ has the desired properties. 
\end{proof}

We now start out with the most important  special case in the proof of Theorem \ref{mainthm21}.

\begin{proposition}\label{prop:mainthm31nongraph}
Let $N$ be an irreducible  $3$-manifold with empty or toroidal boundary that is not a closed graph manifold.
Then there exists a unitary representation that factors through a finite group that makes $N$ acyclic.
\end{proposition}

In the proof of Proposition \ref{prop:mainthm31nongraph} we will need the following elementary lemma.

\begin{lemma}\label{lem:tor}
Let $A(t)$ be an $n\times n$-matrix over $\ct$ and let $z\in \C\sm \{0\}$. We view $\C$ as a $\ct$-module via $t^i\mapsto z^i$. If $z$ is not a zero of $\det(A(t))$, then 
\[ \ct^n/A(t)\ct^n\otimes_{\ct}\C=0 \mbox{ and } \tor_{\ct}(\ct^n/A(t)\ct^n,\C)=0.\]
\end{lemma}

\begin{proof}
It follows from the definitions and standard arguments that  for any $z\in \C\sm \{0\}$ there exists an exact sequence
\[ \ba{rcl}0&\to &\tor_{\ct}(\ct^n/A(t)\ct^n,\C)\to \ct^n\otimes_{\ct}\C\xrightarrow{A(t)\otimes \op{id}}\ct^n\otimes_{\ct}\\[2mm]
&\to &\ct^n/A(t)\ct^n\otimes_{\ct}\C\to 0.\ea\]
The middle map is equivalent to the map $A(z)\colon \C^n\to\C^n$.
In particular, if $\det(A(z))\ne 0$, then the middle map is an isomorphism, hence the two outer modules have to vanish.
\end{proof}

We are now in a position to prove  Proposition \ref{prop:mainthm31nongraph}.

\begin{proof}[Proof of Proposition \ref{prop:mainthm31nongraph}]
Let $N$ be an irreducible  $3$-manifold  with empty or toroidal boundary that is not a closed graph manifold.
It follows from the work of Agol \cite{Ag13}, Liu \cite{Liu13}, Przytycki--Wise \cite{PW14,PW12} and Wise \cite{Wi09,Wi12a,Wi12b}
that $\pi_1(N)$ is `virtually special'. By work of Haglund--Wise \cite{HW08} and Agol \cite{Ag08} this implies that $\pi_1(N)$ is `virtually RFRS'. (We also refer to \cite{AFW13} for precise references.) 
It follows from Agol's virtual fibering theorem (see \cite[Theorem~5.1]{Ag08} and see \cite[Theorem~5.1]{FKt14} for an alternative account of Agol's proof) 
that $N$ is virtually fibered. This means that there exists a finite regular cover $M$ of $N$ which admits the structure of a surface bundle $p\co M\to S^1$.
(For graph manifolds with boundary this conclusion was already proved by Wang--Yu \cite{WY97}.)

We denote  by $\phi$ the homomorphism $p_*\colon \pi_1(M)\to \Z=\pi_1(S^1)$ induced by the projection. 
Note that $\phi$ also gives rise to a ring homomorphism
\[ \ba{rcl} \Z[\pi_1(M)]&\to& \ct\\
\sum_{i=1}^r a_ig_i&\mapsto & \sum_{i=1}^r a_it^{\phi(g_i)}\ea, \]
which we again denote by $\phi$.
Since $\pi_1(M)\to \ll t\rr$ is surjective and since it corresponds to a surface bundle  it follows  from standard arguments (see e.g. \cite{FKm06}) that there exist non-zero polynomials  $p_1(t),\dots,p_r(t)\in \ct$ such that
\[ 
\ba{rcl} H_0(M;\ct)&\cong& \ct/(1-t),\\
 H_1(M;\ct)&\cong &\bigoplus\limits_{i=1}^r \ct/p_i(t),\\
 H_2(M;\ct) &\cong &\left\{ \ba{ll} \ct/(1-t), &\mbox{ if $M$ is closed}, \\ 0, &\mbox{ otherwise.}\ea \right.\\
 H_i(M;\ct)&=&0, \mbox{ for any $i\geq 3$}.\ea \]

Now let $z\ne 1$ be a root of unity that is not a zero of any of the polynomials $p_1(t),\dots,p_r(t)$.
Such a $z$ exists since all the polynomials are non-zero.
We denote by $\rho\co \pi_1(M)\to \gl(1,\C)$ the character that is given by
$g\mapsto z^{\phi(g)}$.
We henceforth view $\C$ as a $(\Z[\pi_1(M)],\C)$-bimodule via $\rho$.
Furthermore, we view $\C$ as a $\ct$-module where the action is induced by $t^i\mapsto z^i$.
Note that as a $\ct$-module we can describe $\C$ also as $\ct/(z-t)$.

In light of Lemma \ref{lem:gotofinitecover} it suffices to prove the following claim.

\begin{claim}
For any $i$ we have $H_i(M;\C_\rho)=0$.
\end{claim}

First we note that there exists a canonical isomorphism
\[ H_i(C_*(\wti{M})\otimes_{\Z[\pi_1(M)]}\C_\rho)\xrightarrow{\cong}
H_i(C_*(\wti{M})\otimes_{\Z[\pi_1(M)]}\ct\otimes_{\ct}\C_\rho).\]
Since  $\ct$ is a PID the universal coefficient theorem applied to the $\ct$-chain complex
$C_*(\ti{M})\otimes_{\Z[\pi_1(M)]}\ct$ and the $\ct$-module $\C_\rho$ 
gives us a short exact sequence
\[ 0\to H_i(M;\ct)\otimes_{\ct}\C_\rho \to H_i(M;\C_\rho)\to \tor_{\ct}(H_{i-1}(M;\ct),\C_\rho)\to 0.\]
It follows from our choice of $z$ and from Lemma \ref{lem:tor} that the two outer modules are zero.
This implies that  the middle module is zero as well.
This concludes the proof of the claim and thus also of the proposition.
\end{proof}

In the previous proposition we used the virtual fibering theorem to show that many $3$-manifolds can be made acyclic. In general this approach will not work for  closed graph manifolds.
For example, it is well-known that non-trivial $S^1$-bundles over surfaces are not virtually fibered.
Furthermore, by \cite[p.~86]{LW93} and \cite[Theorem~D]{Ne96} there exist many graph manifolds with a non-trivial JSJ-decomposition, that are not virtually fibered. Therefore for these manifolds we need  different arguments to show that they can be made acyclic.

\begin{proposition}\label{prop:mainthmspherical}
Let $N\ne S^3$ be a spherical  $3$-manifold. Then there exists a unitary representation  that makes $N$ acyclic.
\end{proposition}

\begin{proof}
Let $N\ne S^3$ be a spherical $3$-manifold. We abbreviate $\pi_1(N)$ with $\pi$.
Since $N$ is spherical there exists  a faithful representation $\a\colon \pi\to SO(4)\subset U(4)$ such that $\pi$ acts freely on $\R^4$. It follows that given any non-trivial $g\in \pi$ the matrix $\a(g)$ does not have $1$ as an eigenvalue. Since $\pi$ is non-trivial there exists a non-trivial element in $\pi$, and it follows from Lemma \ref{lem:h0} that $H_0(N;\C^4_\a)=0$. 

It is an immediate consequence of Maschke's theorem \cite[Theorem~XVIII.1.2]{La11} that there exists an $n$ and a $(\Z[\pi],\C)$-bimodule $V$ such that $\C[\pi]^n=\C^4_{\a}\oplus V$. Since the universal cover of $N$ is $S^3$ it  follows from Lemma \ref{lem:eckmann-shapiro} that for $i=1,2$ we have
\[ H_i(N;\C^4_\a)\oplus H_i(N;V)\cong H_i(N;\C[\pi]^n)=H_i(S^3;\C^n)=0.\]
Finally it follows from Lemma \ref{lem:eulermultiplicative} that $\chi(N,\a)=4\cdot \chi(N)=0$. Together with the above calculations this implies that $H_3(N;\C^4_\a)=0$.
\end{proof}

\begin{proposition}\label{prop:mainthmsfs}
Let $N$ be a  $3$-manifold that is a closed non-spherical Seifert fibered manifold.
Then there exists a unitary representation that factors through a finite group that makes $N$ acyclic.
\end{proposition}

\begin{proof}
Let $N$ be a  $3$-manifold that is a closed non-spherical Seifert fibered manifold. We write $\pi=\pi_1(N)$.
We start out with the following claim.

\begin{claim}
The manifold $N$ is covered by an $S^1$-bundle $\wti{N}$ over an orientable surface such that the $S^1$-fiber defines a  non-trivial element in $H_1(\wti{N};\Z)$.
\end{claim}

We first note that if $M$ is a $3$-manifold that is an $S^1$-bundle over an orientable surface $F$ of genus $g$ with Euler number $e$, then 
it follows from the discussion in 
\cite[Section~8.2]{AFW13} (which builds in turn on \cite[p.~435]{Sc83}) that
\[ \pi_1(M)\cong \ll a_1,b_1,\dots,a_g,b_g,t\,|\, \prod_{i=1}^g [a_i,b_i]=t^e, [a_i,t]=1,[b_i,t]=1\mbox{ for }i=1,\dots,g\rr\]
where $t$ corresponds to the fiber of the $S^1$-bundle. In particular,  the $S^1$-fiber generates a subgroup 
of $H_1(M;\Z)$ that is isomorphic to $\Z/e\Z$.

We now turn to the proof of the claim.
By \cite[Lemma~8.8]{AFW13} there exists a finite cover $N'$ of $N$ that is an $S^1$-bundle over a surface $F'$ of genus $g$. After possibly going to a double cover we can and will assume that $F'$ is orientable.
As above we denote by $e'$ the Euler number of the $S^1$-bundle. If  $e'\ne \pm 1$, then we are done by the above.
Now suppose that $e'=\pm 1$. Since $N$ is not spherical it follows that $\pi_1(N)$ is infinite, in particular the surface $F'$ is not a sphere. Therefore there exists a non-trivial epimorphism $\pi_1(F')\to \Z_2$.
We denote by $\wti{F'}\to F'$ the corresponding $2$-fold cover and we denote by $\wti{N'}\to N'$ the 2-fold cover corresponding to the epimorphism $\pi_1(N')\to\pi_1(F')\to \Z_2$. It follows from \cite[Lemma~3.5]{Sc83} that $\wti{N'}$ is an $S^1$-bundle over $\wti{F'}$ with Euler number $2e'$. It thus follows from the above discussion that we are done.
This concludes the proof of the claim.

The claim together with Equation (5.a) in \cite[Section~VII.5.1]{Tu02} implies that $N$ admits a finite cover that can be made acyclic using a unitary $1$-dimensional representation that factors through a finite group. 
The proposition  follows from  Lemma \ref{lem:gotofinitecover}.
\end{proof}

\begin{proposition}\label{prop:mainthm31graph}
Let $N$ be a closed graph manifold that is not a Seifert fibered manifold. 
Then there exists a unitary representation that factors through a finite group that makes $N$ acyclic.
\end{proposition}

\begin{proof}
Let $N$ be a closed graph manifold that is not a Seifert fibered manifold. 
If $N$ is finitely covered by a torus-bundle, then the argument in the proof of Proposition \ref{prop:mainthm31nongraph} shows that $N$ can be made acyclic using a unitary representation that factors through a finite group. 

Now suppose that  $N$ is not finitely covered by a torus-bundle. It follows from \cite[Proposition~2.9]{Na14} and \cite[Proposition~4.18]{Na14} (applied to the zero cohomology class) that there exists a finite covering $p\colon N'\to N$ and a one-dimensional unitary representation $\a'\colon \pi_1(N')\to U(1)$ that factors through a finite group that makes $N'$ acyclic. The proposition, once again, follows from  Lemma \ref{lem:gotofinitecover}.
\end{proof}

\begin{proposition}\label{prop:mainthmnonprime}
Let $N\ne S^3$ be a $3$-manifold that  is closed or has toroidal boundary
and with at most one prime summand that is not a rational homology sphere.
Then there exists a unitary representation that factors through a finite group that makes $N$ acyclic.
\end{proposition}

\begin{proof}
Let $N$ be a $3$-manifold that  is closed or has toroidal boundary
and which has at most one prime summand that is not a rational homology sphere.
A straightforward Mayer--Vietoris theorem shows that the connected sum of two rational homology spheres is again a rational homology sphere. It follows from the Prime Decomposition Theorem that $N$ can be written as a connected sum $N\cong N_1\# N_2$ where $N_1\ne S^3$ is prime and $N_2$ is a rational homology sphere. For $i=1,2$ we pick an open $3$-ball $D_i\subset N_i$
and we pick an orientation reversing diffeomorphism $f\co \partial \ol{D_1}\to \partial \ol{D_2}$. We then identify $N$ with $(N_1\sm D_1)\cup_{\partial \ol{D_1}=\partial \ol{D_2}}(N_2\sm D_2)$.

By the previous propositions there exists a unitary representation $\a\colon \pi_1(N_1)\to \aut_\C(V)$ that makes $N_1$ acyclic and that factors through a finite group. We equip $N=N_1\# N_2$ with the representation that is given by 
\[ \pi_1(N)=\pi_1(N_1\sm D_1)*\pi_1(N_2\sm D_2)\to \pi_1(N_1\sm D_1)=\pi_1(N_1)\xrightarrow{\a}\aut_\C(V).\]
We denote by $g\colon N_2\sm D_2\to \ol{D_1}$ the map that is given by collapsing the complement of a neighborhood of $\partial \ol{D_2}$ to a point and that  induces the map $f$ on $\partial \ol{D_2}$. This map $g$, together with the identity on $N_1\sm D_1$ also induces a  map $g\colon N\to N_1$. These maps 
give us the following commutative diagram of  Mayer-Vietoris sequences:
\[
\ba{ccccccccccc}
\to &H_i(S^2;V)&\to & H_i(N_1\sm D_1;V)&\oplus&  H_i(N_2\sm D_2;V)&\to &H_i(N;V)&\to \\[2mm]
&\downarrow\id &&\downarrow \id &&\downarrow g_*&&\downarrow g_*& \\[2mm]
\to & H_i(S^2;V)&\to & H_i(N_1\sm D_1;V)&\oplus&  H_i(\ol{D_1};V)&\to &H_i(N_1;V)&\to \ea \]
Since $N_2\sm D_2$ and $\ol{D_1}$ are both rational homology ball we see that the maps $H_i(N_2\sm D_2;\C)\to H_i(\ol{D_1};\C)$ on ordinary $\C$-homology are isomorphisms. But the coefficient system is trivial in both cases, so it follows from Lemma \ref{lem:trivialrep} that the maps $H_i(N_2\sm D_2;V)\to H_i(\ol{D_1};V)$ are also isomorphisms. We deduce from the five-lemma that the maps $H_i(N;V)\to H_i(N_1;V)$ are isomorphisms. Since the latter twisted homology groups vanish, so do the former.
\end{proof}

It is clear that all the propositions above put together give a proof of  Theorem \ref{mainthm21}.

\end{document}